\documentclass[a4paper,12pt, twoside, reqno]{amsart}
\usepackage{amssymb}
\usepackage{amsmath}
\newtheorem{theorem}{Theorem}

\newtheorem{proposition}[theorem]{Proposition}
\newtheorem{definition}{Definition}
\newtheorem{example}{Example}
\newtheorem{remark}{Remark}[section]

\title{Fixed points theorems for unsaturated and saturated classes of contractive mappings in Banach spaces}
\author{Vasile BERINDE}
\address{Department of Mathematics and Computer Science, Technical University of Cluj-Napoca, North University Center at Baia Mare, 
Victoriei 76, 430122 Baia Mare, ROMANIA; 
E-mail: vasile.berinde@mi.utcluj.ro}

\author{M\u ad\u alina P\u acurar}
\address{Department of Economics and Bussiness Administration
	 in German Language 
	Faculty of Economics  and Bussiness Administration 
	Babe\c s-Bolyai University of Cluj-Napoca 
	T. Mihali 58-60, 400591 Cluj-Napoca, ROMANIA; E-mail: madalina.pacurar@econ.ubbcluj.ro}

\begin{document}
\maketitle \pagestyle{myheadings} \markboth{V. Berinde, M. P\u acurar} {Fixed points theorems for unsaturated and saturated classes}
\begin{abstract}
Based on the technique of enriching contractive type mappings, a technique that has been used successfully in some recent papers, we introduce the concept of {\it saturated} class of contractive mappings. We show that, from this  perspective, the contractive type mappings in the metric fixed point theory can be separated into two distinct classes, unsaturated and saturated, and that, for any unsaturated class of mappings, the technique of enriching contractive type mappings provides genuine new fixed point results. We illustrate the concept by surveying some significant fixed point results obtained recently for five remarkable unsaturated classes of contractive mappings. In the second part of the paper, we also identify two important classes of saturated contractive mappings, whose main feature is that they cannot be enlarged by means of the technique of enriching the contractive mappings.
\end{abstract}

\section{Introduction}

Let $(X,d)$ be a metric space and let $T:X\rightarrow X$ be a self mapping. We denote the set of fixed points of $T$ by $Fix\,(T)$, i.e., $Fix\,(T)=\{a\in X: Ta=a\}$ and define the $n^{th}$ iterate of $T$ as usually, that is, $T^0=I$ (identity map) and $T^n=T^{n-1}\circ T$, for $n\geq 1$.

The mapping $T:X\rightarrow X$ is called a {\it Banach contraction} if there exists a constant $c\in[0,1)$ such that
\begin{equation} \label{cond_Banach}
d(Tx,Ty)\leq c\cdot d(x,y),\textnormal{ for all } x,y \in X.
\end{equation}

Recall that the mapping $T$ is said to be a {\it Picard operator}, see for example Rus \cite{Rus01}, if $ (i) $ $Fix\,(T)=\{p\}$  and $ (ii) $  $T^nx_0 \rightarrow p$ as $n\rightarrow \infty$, for any $x_0$ in $X$.

The Banach contraction mapping principle,  first introduced by Banach \cite{Ban22} in the setting of a Banach space and then extended by Caccioppoli  \cite{Cac} to the setting of a complete metric space,  essentially states that, in a complete metric space $(X,d)$, any Banach contraction $T:X\rightarrow X$
is a Picard operator.

It is easy to see by \eqref{cond_Banach} that any Banach contraction is continuous but, as shown by Kannan in 1968, see also \cite{Kan68}, \cite{Kan69}, in general, a Picard operator is not necessarily continuous. As, for two distinct points $x,y\in X$, there exist six possible displacements
\begin{equation}\label{eq1}
d(Tx, Ty), d(x,y), d(x,Tx), d(y,Ty), d(x,Ty) \textnormal{ and } d(y,Tx),
\end{equation}
in order to obtain a metric fixed point result for  a mapping $T$, there were used two, like in the case of the Banach contraction condition \eqref{cond_Banach}, three or more displacements from the list \eqref{eq1}. 

In the development of metric fixed point theory, the following contractive condition, introduced in 1969 by Kannan \cite{Kan68}, \cite{Kan69}, and which uses three displacements from the list \eqref{eq1}, is extremely important for many reasons.

A mapping $T:X\rightarrow X$ is called a {\it Kannan mapping} if there exists a constant $a\in[0,1/2)$ such that
\begin{equation} \label{cond_Kannan}
d(Tx,Ty)\leq a \left(d(x,Tx)+d(y,Ty)\right),\textnormal{ for all } x,y \in X.
\end{equation}

\begin{theorem} (\cite{Kan68})\label{th-K}
Let $(X,d)$ be  a complete metric space and let $T:X\rightarrow X$ be a Kannan mapping. Then $T$ is a Picard operator.
\end{theorem}

In 1972,   Chatterjea \cite{Chat}  introduced  the following dual of the Kannan contraction condition:
\begin{equation} \label{Def_Chatterjea}
d(Tx,Ty)\leq b \left[d(x,Ty)+d(y,Tx)\right],\forall x,y \in X,
\end{equation}
where $b$ is a constant,  $b\in[0,1/2)$, for which a similar fixed point theorem to Theorem \ref{th-K} holds.

It is known, see  \cite{Rho}, \cite{Mes}  and \cite{Ber04}, that conditions \eqref{cond_Banach}, \eqref{cond_Kannan} and \eqref{Def_Chatterjea} are independent, that is, there exist Banach contractions which are not Kannan mappings, and Kannan mappings, which are not Banach contractions and so on. For more details and discussions on this topic see \cite{Mes}, \cite{Rho},  \cite{Ber07} and \cite{Pac09}.

Based on the previous observation, in 1972 Zamfirescu \cite{Zam} formulated a very interesting fixed point theorem that involves all three conditions \eqref{cond_Banach}, \eqref{cond_Kannan}, and \eqref{Def_Chatterjea}.   

\begin{theorem}[\cite{Zam}]\label{thZ}
Let $(X,d)$ be a complete metric space and $T:X\longrightarrow X$ a map for which there exist the real numbers $a,b$ and
$c$ satisfying $0\leq a<1$, $0<b, c<1/2$, such that for each pair $x,y$ in $X$, at least one of the following is true:
\newline
\indent $(z_1)$ $d(Tx, Ty)\leq a\,d(x,y)$;\\[4pt]
\indent $(z_2)$ $d(Tx, Ty)\leq b\big[d(x, Tx)+d(y, Ty)\big]$;\\[4pt]
\indent $(z_3)$ $d(Tx, Ty)\leq c\big[d(x,Ty)+d(y,Tx)\big]$.\newline
Then $T$ is a Picard operator.
\end{theorem}

In 2004, Berinde \cite{Ber04} unified all previous classes of contractive mappings by means of the {\it non symmetric} contraction condition \eqref{2.1} and established an essentially different kind of fixed point theorems, see \cite{BerP16} for a recent survey on the fixed point theory of single-valued almost contractions.

\begin{definition} (\cite{Ber04}) \label{def-1}
Let $(X,d)$ be a metric space. A mapping $T:X\rightarrow X$ is called a $(\delta,L)$-
\textit{almost contraction} if there exist the constants $\delta\in (0,1)$ and $L\geq 0$ such that
\begin{equation}  \label{2.1}
d(Tx, Ty)\leq\delta\cdot d(x,y)+L d(y,Tx)\,, \quad\textnormal{for all}\;\;x,y\in
X.
\end{equation}
\end{definition}

\noindent Another general contraction conditions in this family, which involves all six displacements in the list \eqref{eq1} and due to L. B. \' Ciri\' c, see \cite{Cir74}, also unifies all previous contractive mappings, except for the class of almost contractions \eqref{2.1}. 

A map $T:X\rightarrow X$ is called a \textit{\' Ciri\' c quasi contraction} if there exists $h\in (0,1)$  such that for all $x,y\in X$,
\begin{equation}  \label{eq-Ciric}
d(Tx, Ty)\leq h\max\big\{d(x,y), d(x, Tx), d(y, Ty), d(x, Ty), d(y,Tx)
\big\}. 
\end{equation}
It is important to note that \eqref{2.1} and \eqref{eq-Ciric} are independent to each other, see \cite{Ber04} and \cite{BerP16}. For the impressive rich literature on this area, we refer to the monographs \cite{Ber07}, \cite{Pac09}, \cite{Rus79}-\cite{Rus01}, \cite{Rus08}, and the references therein.

    On the other hand, in some recent papers \cite{Ber19b},  \cite{Ber19a},  \cite{Ber20},  \cite{BerP20},  \cite{BerP21},  \cite{BerP19c} and \cite{arhiv}, the authors used the technique of enrichment of contractive type mappings to generalize, in the setting of a Hilbert or a Banach space, well known and important classes of contractive type mappings  from the metric fixed point theory satisfying {\it symmetric} and {\it non symmetric} metric conditions. 

The technique of enriching the contractive mappings has been suggested by the following development in fixed point theory related to the concept of asymptotic regularity. Formally, this  notion was introduced in 1966 by Browder and Petryshyn (\cite{BroP66}, Definition 1, page 572) in connection with the study of fixed points of nonexpansive mappings but the same property was used in 1955 by Krasnosel'ski\v \i  \, \cite{Kra55},  to prove that if $K$ is a compact convex subset of a uniformly convex Banach space and if $T: K\rightarrow K$ is nonexpansive, then, for any $x_0\in K$, the sequence
\begin{equation} \label{eq1a}
x_{n+1}=\frac{1}{2}\left(x_n+ T x_n\right),\,n\geq 0,
\end{equation}
converges to to a fixed point of $T$. 

In proving his result, Krasnosel'ski\v\i  \, used the important fact that, if $T$ is a nonexpansive mapping, which in general is not asymptotically regular, then the averaged mapping involved in \eqref{eq1a}, that is, $\frac{1}{2}I+\frac{1}{2} T$, is asymptotically regular. For the general averaged mapping 
$$
T_\lambda:=(1-\lambda) I+\lambda T, \, \lambda\in (0,1),
$$
and in the setting of a Hilbert space, the corresponding result has been stated by Browder and Petryshyn  (\cite{BroP67}, Corollary to Theorem 5), while Ishikawa \cite{Ishi} proved in 1976  the same result with no restriction on the geometry of the Banach space involved. 

Therefore, the averaged operator $T_\lambda$ is {\em enriching} the class of nonexpansive mappings with respect to the asymptotic regularity.
This fact suggested us that one could similarly enrich the classes of contractive mappings in metrical fixed point theory by imposing that $T_\lambda$ rather than $T$ satisfies a certain contractive condition.

In this way, the following classes of mappings were introduced and studied: enriched contractions and enriched $\varphi$-contractions \cite{BerP20}, enriched Kannan contractions \cite{BerP21}, enriched Chatterjea mappings \cite{BerP19c}, enriched almost contractions \cite{arhiv}, enriched nonexpansive mappings in Hilbert spaces \cite{Ber19a}, enriched nonexpansive mappings in Banach spaces \cite{Ber19a},  enriched strictly pseudocontractive operators \cite{Ber19b}, enriched nonexpansive semigroups \cite{Kes} etc. 

We note that all the above mentioned fixed point results for enriched mappings were established in Banach or Hilbert spaces, which are metric spaces with generous geometrical properties. In order to extend those fixed point results  to the setting of a metric space $(X,d)$, one needs some additional geometric properties of the space $X$, related to the convexity in the usual sense for subsets of Euclidian space, and expressed by the fact that, for any two distinct points $x$ and $y$ in $X$, there exists a third point $z$ in $X$ {\it lying between} $x$ and $y$.

Using an appropriate convexity structure in metric spaces, introduced by Takahashi \cite{Tak}, some of the above-mentioned results for enriched contractions, that is, those for enriched contractions and enriched $\varphi$-contractions \cite{BerP20}, have been subsequently extended \cite{BerP21} to the more general setting of convex metric spaces.

Moreover, after a close examination of the new fixed point results reported in \cite{Ber19b},  \cite{Ber19a},  \cite{Ber20},  \cite{BerP20},  \cite{BerP21},  \cite{BerP19c} and \cite{arhiv}, we noted that the technique of enriching contractive type mappings $T$ by means of the averaged operator $T_\lambda$ cannot effectively enlarge  all classes of contractive mappings, as initially expected.

This observation suggested us a new interesting concept, that is, that of {\it saturated class of contractive mappings} with respect to the averaged operator $T_\lambda$,  which will be introduced in the next section. 

We show that,  from this  perspective, many well-known contractive type mappings in metrical fixed point theory can be classified into these two distinct classes: unsaturated contractive mappings and saturated contractive mappings. 

For  some remarkable unsaturated classes of contractive mappings, we present various significant fixed point results with appropriate illustrative examples, while, in the second part of the paper, we also identify two important classes of saturated contractive mappings, whose main feature is that they cannot be enlarged by means of the technique of enriching contractive mappings.

\section{Unsaturated classes of contractive mappings}

In this section we present five classes of unsaturated contractive mappings: Banach contractions, Kannan contractions, Chatterjea contractions, almost contractions and nonexpansive mappings. For each class, after showing why it is unsaturated, we also indicate some relevant fixed point results obtained by the technique of enriching the respective class. 

First we need to recall an important property that is fundamental in obtaining all fixed point results for enriched contractive mappings reported in \cite{Ber19b},  \cite{Ber19a},  \cite{Ber20},  \cite{BerP20},   \cite{BerP19c} and \cite{arhiv}.

	Let $T:C\rightarrow C$ be a self mapping of a convex subset $C$ of a linear space $X$. Then, for any  $\lambda\in(0,1]$, the so-called averaged mapping (a term coined in \cite{Bai}) $T_{\lambda}$ given by
	\begin{equation} \label{def_T.lambda}
	T_\lambda x=(1-\lambda)x+\lambda Tx, \forall  x \in C,
	\end{equation}
	has the property that 
\begin{equation} \label{average}
Fix(\,T_\lambda)=Fix\,(T).
\end{equation}


\begin{definition}\label{bogat}
Let $(X,\|\cdot\|)$ be a linear normed space and let $\mathcal{C}$ be a subset of the family of all self mappings of $X$. A mapping $T:X\rightarrow X$ is said to be  $\mathcal{C}${\em -enriched} or {\em enriched with respect to} $\mathcal{C}$ if there exists $\lambda\in (0,1]$ such  $T_\lambda\in \mathcal{C}$. 

We denote by $\mathcal{C}^{e}$ the set of all enriched mappings with respect to $\mathcal{C}$.
\end{definition}

\begin{remark}
From Definition \ref{bogat} it immediately follows that $\mathcal{C}\subseteq \mathcal{C}^{e}$.
\end{remark}

\begin{definition}\label{sat}
Let $\mathcal{C}$ be a subset of the family of all self mappings of $X$. We say that $\mathcal{C}$  is an {\em unsaturated class of mappings} if the inclusion $\mathcal{C}\subset \mathcal{C}^{e}$ is strict, while, if $\mathcal{C}= \mathcal{C}^{e}$, then we say that $\mathcal{C}$ is a {\em saturated} class of mappings.
\end{definition}

The main aim of this section is to present some important examples of unsaturated classes of contractive mappings, each of them with a relevant new fixed point result.

\subsection{Banach contractions} 

Let $(X,d)$ be a metric space and let $\mathcal{C}_{B}(X)$ denote the class of Banach contractions defined on $X$, that is, the class of all mappings which satisfy  a contractive condition of the form \eqref{cond_Banach} on $X$. The concept of {\it enriched Banach contraction} has been introduced and studied in \cite{BerP20}. 

\begin{definition} (Definition 2.1, \cite{BerP20})\label{def1}
Let $(X,\|\cdot\|)$ be a linear normed space. A mapping $T:X\rightarrow X$ is said to be a $(b,\theta$)-{\it enriched contraction} if there exist $b\in[0,+\infty)$ and $\theta\in[0,b+1)$ such that
\begin{equation} \label{eq3}
\|b(x-y)+Tx-Ty\|\leq \theta \|x-y\|,\forall x,y \in X.
\end{equation}
\end{definition}

Obviously, any Banach contraction satisfies \eqref{eq3} with $b=0$.

\begin{proposition} \label{p1}
Let $(E,\|\cdot,\cdot\|)$ be a Banach space. Then $\mathcal{C}_{B}(E)$ is an unsaturated class of contractive mappings.
\end{proposition}

\begin{proof} 
Denoting $\lambda=\dfrac{1}{b+1}$, by \eqref{eq3} we obtain 
$$
\|T_\lambda x-T_\lambda y\|\leq c \cdot \|x-y\|,\forall x,y \in X,
$$
which shows that $T_\lambda\in \mathcal{C}_{B}(E)$, where we denoted $c=\lambda \theta<1$.

The fact that the inclusion $\mathcal{C}_{B}(X)\subset \mathcal{C}^{e}_{B}(X)$ is strict, and hence $\mathcal{C}_{B}(X)$ is unsaturated, follows by the next example.  

\end{proof}

\begin{example}  \label{ex1}
Let $X=[0,1]$ be endowed with the usual norm and let $T:X\rightarrow X$ be defined by $Tx=1-x$, for all $x\in [0,1]$. Then $T\notin \mathcal{C}_{B}(X)$, just take  $x=0$, $y=1$ in \eqref{cond_Banach} to get the contradiction $1\leq c<1$.

But $T\in \mathcal{C}^{e}_{B}(X)$ as  $T_\lambda\in \mathcal{C}_{B}(X)$ for any $\lambda\in(0,\dfrac{1}{2})$. Note also that $Fix\,(T)=\left\{\dfrac{1}{2}\right\}$.
\end{example}

Proposition \ref{p1} now guaranties that the next theorem, the main result in \cite{BerP20}, is a genuine generalization of the Banach contraction principle in the setting of Banach spaces.

\begin{theorem} \label{thB}
Let $(X,\|\cdot\|)$ be a Banach space and $T:X\rightarrow X$ a $(b,\theta$)-{\it enriched contraction}. Then

$(i)$ $Fix\,(T)=\{p\}$, for some $p\in X$;

$(ii)$ There exists $\lambda\in (0,1]$ such that the iterative method
$\{x_n\}^\infty_{n=0}$, given by
$$
x_{n+1}=(1-\lambda)x_n+\lambda T x_n,\,n\geq 0,
$$
converges to p, for any $x_0\in X$;

$(iii)$ The following estimate holds
$$
\|x_{n+i-1}-p\| \leq\frac{c^i}{1-c}\cdot \|x_n-
x_{n-1}\|\,,\quad n=0,1,2,\dots;\,i=1,2,\dots,
$$
where $c=\dfrac{\theta}{b+1}$.
\end{theorem}

\subsection{Kannan contractions} 

Let $(X,d)$ be a metric space and let $\mathcal{C}_{K}(X)$ denote the class of Kannan contractions defined on $X$, that is, the class of all mappings which satisfy  a contractive condition of the form \eqref{cond_Kannan} on $X$. The concept of {\it enriched Kannan contraction} has been introduced and studied in \cite{BerP21}.

\begin{definition} (Definition 2.1, \cite{BerP21})\label{def2}
Let $(X,\|\cdot\|)$ be a linear normed space. A mapping $T:X\rightarrow X$ is said to be a $(k,a$)-{\it enriched Kannan mapping} if there exist  $a\in[0,1/2)$ and $k\in[0,\infty)$ such that
\begin{equation} \label{cond_eKannan}
\|k(x-y)+Tx-Ty\|\leq a \left(\|x-Tx\|+\|y-Ty\|\right),\textnormal{ for all } x,y \in X.
\end{equation}
\end{definition}

Obviously, any Kannan mapping satisfies \eqref{cond_eKannan} with $k=0$.

\begin{proposition} \label{p2}
Let $(X,\|\cdot,\cdot\|)$ be a Banach space. Then $\mathcal{C}_{K}(X)$ is an unsaturated class of contractive mappings.
\end{proposition}

\begin{proof}
Take $\lambda=\dfrac{1}{k+1}\in(0,1]$. Then $k=1/\lambda-1$ and thus the contractive condition \eqref{cond_eKannan} becomes
$$
\left \|\left(\frac{1}{\lambda}-1\right)(x-y)+Tx-Ty\right\|\leq a \left(\|x-Tx\|+\|y-Ty\|\right),\textnormal{ for all } x,y \in X,
$$
which can be written in an equivalent form as
$$
\|T_\lambda x-T_\lambda y\|\leq a  \left(\|x-T_\lambda x\|+\|y-T_\lambda y\|\right),\textnormal{ for all } x,y \in X,
$$
and this shows that $T_\lambda$ is a Kannan mapping. To prove that the inclusion $\mathcal{C}_{K}(X)\subset \mathcal{C}^{e}_{K}(X)$ is in general strict, we use the next example.
\end{proof}

\begin{example}  \label{ex2}
Let $X$ and $T$ be as in Example \ref{ex1}. Then, see Example 2.1 in \cite{BerP21}, for any $a\in[0,1/2)$, $T$ is a $(1-2a,a)$-enriched Kannan mapping, which means that for $\lambda=\dfrac{1}{2-2a}$, $T_\lambda\in \mathcal{C}_{K}(X)$. 
\end{example}

Proposition \ref{p2} now guaranties that the next theorem, the main result in \cite{BerP21}, is a genuine generalization of the Kannan fixed point theorem in the setting of Banach spaces.

\begin{theorem} \label{thK}
Let $(X,\|\cdot\|)$ be a Banach space and $T:X\rightarrow X$ a $(k,a$)-{\it enriched Kannan mapping}. Then

$(i)$ $Fix\,(T)=\{p\}$, for some $p\in X$;

$(ii)$ There exists $\lambda\in (0,1]$ such that the iterative method
$\{x_n\}^\infty_{n=0}$, given by
$$
x_{n+1}=(1-\lambda)x_n+\lambda T x_n,\,n\geq 0,
$$
converges to p, for any $x_0\in X$;

$(iii)$ The following estimate holds
$$
\|x_{n+i-1}-p\| \leq\frac{\delta^i}{1-\delta}\cdot \|x_n-
x_{n-1}\|\,,\quad n=0,1,2,\dots;\,i=1,2,\dots
$$
where $\delta=\dfrac{a}{1-a}$.

\end{theorem}

\subsection{Chatterjea contractions} 

Let $(X,d)$ be a metric space and let $\mathcal{C}_{C}(X)$ denote the class of Chatterjea contractions defined on $X$, that is, the class of all mappings which satisfy  a contractive condition of the form \eqref{Def_Chatterjea} on $X$. The concept of {\it enriched Chatterjea contraction} has been introduced and studied in \cite{BerP19c}.

\begin{definition} (Definition 2.1, \cite{BerP19c})\label{def3}

Let $(X,\|\cdot\|)$ be a linear normed space. A mapping $T:X\rightarrow X$ is said to be an $(k,b$)-{\it enriched Chatterjea mapping} if there exist $b\in[0,1/2)$ and $k\in[0,+\infty)$  such that
$$
\|k(x-y)+Tx-Ty\|\leq  b \left[\|(k+1)(x-y)+y-Ty\|+\right.
$$
\begin{equation} \label{Def_eChatterjea}
\left.+\|(k+1)(y-x)+x-Tx\|\right]\|,\forall x,y \in X.
\end{equation}
\end{definition}

Obviously, any Chatterjea mapping satisfies \eqref{Def_eChatterjea} with $k=0$.

\begin{proposition} \label{p3}
Let $(X,\|\cdot,\cdot\|)$ be a Banach space. Then $\mathcal{C}_{C}(X)$ is an unsaturated class of contractive mappings.
\end{proposition}

\begin{proof}
If we take $\lambda=\dfrac{1}{k+1}$, then  we have $0<\lambda<1$ and by the contractive condition \eqref{Def_eChatterjea} we obtain
$$
\|T_\lambda x-T_\lambda y\|\leq b  \left[\|x-T_\lambda y\|+\|y-T_\lambda x\|\right],\forall x,y \in X,
$$
with $b\in[0,1/2)$, which shows that $T_\lambda\in \mathcal{C}_{C}(X)$. 

To prove that the inclusion $\mathcal{C}_{C}(X)\subset \mathcal{C}^{e}_{C}(X)$ is general strict, we use the next example.
\end{proof}

\begin{example}  \label{ex3}
Let $X=[0,1]$ be endowed with the usual norm and $T:X\rightarrow X$ be defined by $Tx=1-x$, for all $x\in [0,1]$. If $T$ would be a Chatterjea mapping, then, in view of \eqref{Def_Chatterjea}, there would exist $b\in[0,1/2)$ such that for all $x,y \in [0,1]$,
$$
|x-y|\leq 2b \cdot |x+y-1|,
$$
which, for $x=0$ and $y=1$, yields the contradiction $1\leq 0$. 

So, $T$ is  not a  Chatterjea mapping but it is an enriched Chatterjea mapping. 
Indeed, the enriched Chatterjea  condition \eqref{Def_eChatterjea} is in this case equivalent to
\begin{equation} \label{Def_Chatterjeab}
|(k-1)(x-y)|\leq b \left[|(k+1)x-(k-1)y-1|+|(k+1)y-(k-1)x-1|\right].
\end{equation} 
Having in view that
$$
2k |x-y|=|[(k+1)x-(k-1)y-1]-[(k+1)y-(k-1)x-1]|
$$
$$
\leq |(k+1)x-(k-1)y-1|+|(k+1)y-(k-1)x-1|,
$$
in order to have \eqref{Def_Chatterjeab} satisfied for all $x,y\in [0,1]$, it is necessary to have $\dfrac{|k-1|}{2k}\leq b $, for a certain $b \in[0,1/2)$.

The only possibility is to have $k<1$ when, by taking $\dfrac{1-k}{2k}= b$, one obtains $k=\dfrac{1}{b+2}$. Therefore, for any $b\in[0,1/2)$ and $\lambda=\dfrac{b+2}{b+3}$, $T_\lambda\in \mathcal{C}_{C}(E)$.
\end{example}

Proposition \ref{p3} now guaranties that the next theorem, the main result in \cite{BerP19c}, is a genuine generalization of the Chatterjea fixed point theorem in the setting of Banach spaces.

\begin{theorem} \label{thC}
Let $(X,\|\cdot\|)$ be a Banach space and $T:X\rightarrow X$ a $(k,b$)-{\it enriched Chatterjea mapping}. Then

$(i)$ $Fix\,(T)=\{p\}$;

$(ii)$ There exists $\lambda\in (0,1]$ such that the iterative method
$\{x_n\}^\infty_{n=0}$, given by
$$
x_{n+1}=(1-\lambda)x_n+\lambda T x_n,\,n\geq 0,
$$
converges to p, for any $x_0\in X$;

$(iii)$ The following estimate holds
$$
\|x_{n+i-1}-p\| \leq\frac{\delta^i}{1-\delta}\cdot \|x_n-
x_{n-1}\|\,,\quad n=0,1,2,\dots;\,i=1,2,\dots
$$
where $\delta=\dfrac{b}{1-b}$.

\end{theorem}

\subsection{Almost contractions}

Let $(X,d)$ be a metric space and let $\mathcal{C}_{AC}(X)$ denote the class of almost contractions defined on $X$, that is, the class of all mappings which satisfy  a contractive condition of the form \eqref{2.1} on $X$. The concept of {\it enriched almost contraction} has been introduced and studied in \cite{arhiv}. 

\begin{definition} (Definition 3, \cite{arhiv})\label{def4}
Let $(X,\|\cdot\|)$ be a linear normed space. A mapping $T:X\rightarrow X$ is said to be an {\it enriched} $(b,\theta, L)$-{\it  almost contraction} if there exist $b\in[0,\infty)$, $\theta\in (0,b+1)$ and $L\geq 0$ such that
\begin{equation} \label{eq3a}
\|b(x-y)+Tx-Ty\|\leq \theta \|x-y\|+L\|b(x-y)+Tx-y\|,\,\forall x,y\in X.
\end{equation}
for all $x,y \in X$.
\end{definition} 
Obviously, any almost contraction satisfies \eqref{eq3a} with $b=0$.
\begin{proposition} \label{p4a}
Let $(X,\|\cdot,\cdot\|)$ be a Banach space. Then $\mathcal{C}_{AC}(X)$ is an unsaturated class of contractive mappings.
\end{proposition}

\begin{proof}
We put  $\lambda=\dfrac{1}{b+1}$. Then $0<\lambda<1$ and thus the contractive condition \eqref{eq3a} can be written equivalently in the form
$$
\|T_\lambda (x)-T_\lambda y\|\leq \delta \|x-y\|+L \|T_\lambda x-y\| ,\forall x,y \in X,
$$
where $\delta=\dfrac{\theta} {b+1}\in (0,1)$, and this shows that $T_\lambda\in \mathcal{C}_{AC}(X)$.
 
To prove that the inclusion $\mathcal{C}_{AC}(X)\subset \mathcal{C}^{e}_{AC}(X)$ is in general strict, we use the next example.
\end{proof}

\begin{example} (Example 3, \cite{arhiv}) \label{ex4}
Let $X=\left[0,\frac{4}{3}\right]$ with the usual norm and $T: X\rightarrow X$ be given by
$$
Tx=
\begin{cases}
1-x, \textnormal{ if } x\in \left[0,\frac{2}{3}\right)\\
\smallskip
2-x, \textnormal{ if } x\in \left[\frac{2}{3}, \frac{4}{3}\right].
\end{cases}
$$
Then, see Example 3 in \cite{arhiv} for details, $T\notin \mathcal{C}_{AC}(X)$ but $T_\lambda\in \mathcal{C}_{AC}(X)$ and hence $T\in \mathcal{C}^{e}_{AC}(X)$.

Note that in this case $Fix\,(T)=\left\{\frac{1}{2},1\right\}$.
\end{example}

Proposition \ref{p4a} now guaranties that the next theorem, the main result in \cite{arhiv}, is a genuine generalization of Theorem 1 \cite{Ber04}, in the setting of Banach spaces.

\begin{theorem} \label{thAC}
Let $(X,\|\cdot\|)$ be a Banach space and let $T:X\rightarrow X$ be a $(b,\theta,L)$-almost contraction.

Then

$1)$ $Fix\,(T)\neq\emptyset$;

$2)$ For any $x_0\in X$, there exists $\lambda\in (0,1)$ such that the Krasnoselkij iteration 
$\{x_n\}^\infty_{n=0}$, defined by  
$$
x_{n+1}=(1-\lambda)x_n+\lambda T x_n,\,n\geq 0,
$$
converges to some $x^\ast\in Fix\,(T)$, for any $x_0\in X$;

$3)$ The following estimate holds
$$
\|x_{n+i-1}-x^\ast\| \leq\frac{\delta^i}{1-\delta}\, \|x_n-
x_{n-1}\|\,,\quad n=0,1,2,\dots;\,i=1,2,\dots,
$$
where $\delta=\dfrac{\theta} {b+1}$.

\end{theorem}

\begin{remark}
We note that Theorem \ref{thAC} unifies Theorem \ref{th-K}, Theorem \ref{thZ}, Theorem \ref{thB}, Theorem \ref{thK}, Theorem \ref{thC} and many other important fixed point theorems in metrical fixed point theory. 
\end{remark}

\subsection{Nonexpansive mappings}

Let $(X,d)$ be a linear normed space and let $\mathcal{C}_{NE}(X)$ denote the class of nonexpansive self mappings defined on $X$, that is, the class of all mappings which satisfy the 
contractive condition 
\begin{equation}\label{nonexp}
\|Tx-Ty\|\leq \|x-y\|,\,\textnormal { for all } x,y\in X.
\end{equation}
The concept of {\it enriched nenexpansive} mapping has been introduced and studied in \cite{Ber19a} in the setting of Hilbert spaces and in \cite{Ber20} in the setting of Banach spaces. 

\begin{definition} (Definition 2.1, \cite{Ber19a})\label{def5}
Let $(X,\|\cdot\|)$ be a linear normed space. A mapping $T:X\rightarrow X$ is said to be a $b$-{\it enriched nonexpansive mapping} if there exists $b\in[0,\infty)$  such that
\begin{equation} \label{eq3u}
\|b(x-y)+Tx-Ty\|\leq (b+1) \|x-y\|,\forall x,y \in X.
\end{equation}
\end{definition} 
Obviously, any nonexpansive mapping satisfies \eqref{eq3u} with $b=0$.

\begin{proposition} \label{p4}
Let $H$  be a Hilbert space. Then $\mathcal{C}_{NE}(H)$ is an unsaturated class of contractive mappings.
\end{proposition}

\begin{proof}
By putting $\lambda=\dfrac{1}{b+1}$, we obtain from \eqref{eq3u} that $T_\lambda\in \mathcal{C}_{NE}(H)$.

To prove that the inclusion $\mathcal{C}_{NE}(H)\subset \mathcal{C}^{e}_{NE}(H)$ is in general strict, we use the next example.
\end{proof}

\begin{example} (Example 2.1, \cite{Ber19a}) \label{ex5}
Let $X=\left[\dfrac{1}{2},2\right]$ be endowed with the usual norm and $T:X\rightarrow X$ be defined by $Tx=\dfrac{1}{x}$, for all $x\in \left[\dfrac{1}{2},2\right]$. Then 
$T\notin \mathcal{C}_{NE}(X)$ but $T_\lambda\in \mathcal{C}_{NE}(X)$ and hence $T\in \mathcal{C}^{e}_{NE}(X)$.

\end{example}

Proposition \ref{p4} now guaranties that the next theorem, the main result in \cite{Ber19a}, is a genuine generalization of  Lemma 3 of Petryshyn \cite{Pet} and of its global variant (Theorem 6) in Browder and Petryshyn \cite{BroP67}, in the setting of Hilbert spaces.

\begin{theorem} (\cite{Ber19a})\label{thNE}
Let $C$ be a bounded closed convex
subset of a Hilbert space $H$ and  $T:C\rightarrow C$ be a $b$-enriched nonexpansive and demicompact mapping. Then the set $Fix\,(T)$ of fixed points of $T$ is a nonempty convex set and there exists $\lambda\in \left(0,1\right)$ such that, for any given $x_0\in C$,  the Krasnoselskij iteration $\{x_n\}_{n=0}^{\infty}$ given by
$$
x_{n+1}=(1-\lambda) x_n+\lambda Tx_n,\,n\geq 0,
$$
converges strongly to a fixed point of $T$.

\end{theorem}

\begin{remark}
The corresponding result to Theorem \ref{thNE} in the setting of Banach spaces has been established in \cite{Ber20}. 
\end{remark}

\section{Saturated classes of contractive mappings}

We first note the following important property emphasized in the previous section: if $\mathcal{C}$ is a saturated class of mappings, that is,
$$
\mathcal{C}=\mathcal{C}^{e},
$$
then there is no any $T\notin \mathcal{C}$, such that $T_{\lambda}\in  \mathcal{C}^{e}$. This expresses the fact that 
$$
\mathcal{C}^{e}=T_{\lambda} ( \mathcal{C})
$$
and means that, for any  $T\in \mathcal{C}$ and $\lambda \in (0,1]$, we have $T_{\lambda}\in  \mathcal{C}$ and, moreover,  $T_\lambda ( T_\mu) \in  \mathcal{C}$, for any $\lambda,\mu \in (0,1]$.

As a direct consequence, the classes of enriched mappings presented in the previous section, i.e., the class of enriched Banach contractions,  $\mathcal{C}_{B}(X)$, the class of Kannan contractions, $\mathcal{C}_{K}(X)$, the class of enriched Chatterjea contractions, $\mathcal{C}_{C}(X)$, the class  of enriched almost contractions, $\mathcal{C}_{AC}(X)$ and the class of enriched nonexpansive mappings, $\mathcal{C}_{NE}(X)$, are all examples of {\it saturated classes of mappings}. 

No significant new fixed point result can be obtained by applying the technique of enrichment to any of these classes of mappings.

Now we present two important saturated classes of contractive mappings in the metric fixed point theory:  strictly pseudo-contractive mappings and demicontractive mappings.

\subsection{Strictly pseudo-contractive mappings} 

Let $C$ be a nonempty subset of a normed space $X$. A mapping  $T:C\rightarrow C$ is called $k${\it -strictly pseudocontractive} if there exists $k\in (0,1)$ such that
\begin{equation}\label{eq-SPC}
\|Tx-Ty\|^2\leq \|x-y\|^2+k\|x-y-(Tx-Ty)\|^2,\forall x,y \in C.
\end{equation}

It is easy to see that any nonexpansive mapping is strictly pseudocontractive but the reverse is not more true, see Example \ref{ex6}. Denote by $\mathcal{C}_{SPC}(X)$ the class of strictly pseudocontractive mappings on $X$. 

If \eqref{eq-SPC} holds with $k=1$, then $T$ is called {\it pseudocontractive}. Obviously, any strictly pseudocontractive mapping is  pseudocontractive but the converse is not true in general.

\begin{example} \label{ex6}
Let $X=\left[\dfrac{1}{2},2\right]$ be endowed with the usual norm and $T:X\rightarrow X$ be defined by $Tx=\dfrac{1}{x}$, for all $x\in \left[\dfrac{1}{2},2\right]$. Then 
condition \eqref{eq-SPC} becomes
$$
\left | \frac{1}{x}-\frac{1}{y}\right |^2\leq |x-y|^2+k\cdot \left |x-y-\left (\frac{1}{x}-\frac{1}{y}\right )\right |^2,
$$
which holds for all $x,y\in \left[\dfrac{1}{2},2\right]$ if $k\in \left (\dfrac{3}{5},1\right)$. Hence $T\in \mathcal{C}_{SPC}(X)$ but $T\notin \mathcal{C}_{NE}(X)$. Indeed, just take $x=1$ and $y=\dfrac{1}{2}$ in \eqref{nonexp} to reach to the contradiction $1\leq \dfrac{1}{2}$.
\end{example}

\begin{theorem} \label{p5}
Let $H$  be a Hilbert space. Then $\mathcal{C}_{SPC}(H)$ is a saturated class of contractive mappings.
\end{theorem}

\begin{proof}
It is enough to show that $\mathcal{C}_{SPC}(H)=\mathcal{C}^{e}_{NE}(H)$. Let $T\in \mathcal{C}_{SPC}(H)$. Then  $T$ satisfies \eqref{eq-SPC} with $k\in (0,1)$.
But
$$
\|x-y-(Tx-Ty)\|^2=\langle x-y-(Tx-Ty), x-y-(Tx-Ty)\rangle>
$$
$$
=\|x-y\|^2-2\langle x-y, Tx-Ty\rangle+\|Tx-Ty\|^2
$$
and so \eqref{eq-SPC} is equivalent to
$$
\|Tx-Ty\|^2\leq \frac{1+k}{1-k}\cdot \|x-y\|^2-\frac{2k}{1-k}\cdot \langle x-y,Tx-Ty\rangle
$$
By adding  in both sides of the previous inequality the quantity
$$
b^2\cdot \|x-y\|^2+2b\langle x-y,Tx-Ty\rangle
$$
we deduce that \eqref{eq-SPC} is equivalent to
\begin{equation}\label{destept}
\|b(x-y)+Tx-Ty\|^2\leq \left(b^2+\frac{1+k}{1-k}\right) \|x-y\|^2+\left(2b-\frac{2k}{1-k}\right)\cdot \langle x-y,Tx-Ty\rangle.
\end{equation}
Now, by taking $b=\dfrac{k}{1-k}$, it follows that $b^2+\dfrac{1+k}{1-k}=(b+1)^2$ and therefore the inequality \eqref{destept} is equivalent to 
$$
\|b(x-y)+Tx-Ty\|\leq (b+1) \|x-y\|,\,\forall x,y \in C, 
$$
which shows that $T\in \mathcal{C}^{e}_{NE}(H)$. 

The converse follows by the fact that all transformations above are equivalent.

\end{proof}

\begin{remark}
Proposition \ref{p5} shows that Theorem 2.1 in \cite{Ber19b} is not a genuine extension of Theorem 12 in Browder and Petryshyn \cite{BroP67}. However, in the case of a Banach space $X$, the class of enriched nonexpansive mappings, $\mathcal{C}^{e}_{NE}(X)$ does not coincide with the class of strictly pseudo-contractive mappings $\mathcal{C}_{SPC}(X)$. Therefore, Theorem  3.2 in \cite{Ber20} is a genuine generalization of Theorem 12 in Browder and Petryshyn \cite{BroP67} and also of Theorem 1 in Senter and Dotson \cite{SeD}, which in turn, is an extension of Theorem 12 in Browder and Petryshyn \cite{BroP67}.
\end{remark}

\subsection{Demicontractive mappings} 

Let $C$ be a nonempty subset of a normed space $X$. A mapping  $T:C\rightarrow C$ is called $k${\it -demicontractive} if $Fix\,(T)\neq \emptyset$ and there exists $k<1$ such that
\begin{equation}\label{eq-DC}
\|Tx-y^*\|^2\leq \|x-y^*\|^2+k\|x-Tx\|^2,\forall x \in C, \forall y^* \in Fix\,(T).
\end{equation}

$T$ is said to be {\it quasi-nonexpansive} if $Fix\,(T)\neq \emptyset$ and 
\begin{equation}\label{eq-QNE}
\|Tx-y^*\|\leq \|x-y^*\|,\forall x \in C, \forall y^* \in Fix\,(T).
\end{equation}

It is easy to see that every quasi-nonexpansive mapping and every nonexpansive mapping with $Fix\,(T)\neq \emptyset$ is demicontractive but the reverses are not valid. 

Note also that if we take $y=y^*\in Fix\,(T)$ in \eqref{eq-SPC}, one obtains exactly condition \eqref{eq-DC}. This shows that any strictly pseudocontractive mapping is also demicontractive but the reverse is not true in general. 

Denote by $\mathcal{C}_{DC}(X)$ and $\mathcal{C}_{QNE}(X)$  the class of demicontractive mappings and quasi-nonexpansive mappings on $X$, respectively.

Recall that the class of demicontractive mappings has been introduced and studied independently in 1977 by M\u aru\c ster \cite{Mar77}, see also \cite{Mar73},  and Hicks and Kubicek \cite{Hic}.

\begin{example} \label{ex7}
Let $H$ be the real line and $C = [0, 1]$. Define $T$ on $C$ by $Tx = \dfrac{2}{3} x \sin \dfrac{1}{x}$, if $x\neq  0$ and $T 0 = 0$. Then, $Fix\,(T)=\{0\}$, $T$ is demicontractive (and also quasi nonexpansive) but $T$ is not nonexpansive. Indeed, for $x\in C$ and $y=0$,
$$
|Tx-0|^2 = |Tx|^2 = \left |\frac{2}{3}x\sin\left (1/x\right )\right |^2 \leq \left |\frac{2}{3} x\right|^2 \leq |x|^2 \leq |x-0|^2 +k|Tx-x|^2,
$$
 for any $k < 1$. Hence \eqref{eq-DC} is satisfied. 
 
 To see that $T$ is not nonexpansive, just take $x = \dfrac{2}{\pi}$ and $y = \dfrac{2}{3\pi}$ to get 
 $$|Tx - T y| = \frac{16}{9\pi} > \frac{4}{3\pi}=|x-y|.
 $$
 Note also that $T$ is not pseudocontractive, hence not strictly pseudocontractive, too. To see that, we take the same values $x = \dfrac{2}{\pi}$ and $y = \dfrac{2}{3\pi}$ as above to get 
 $$|Tx - T y|^2 = \frac{256}{81\pi^2}>\frac{160}{81\pi^2}= |x-y|^2+|(x-y)-(Tx-Ty)|^2.$$
\end{example}

\begin{theorem} \label{p9}
Let $H$  be a Hilbert space. Then $\mathcal{C}_{DC}(H)$ is a saturated class of contractive mappings.
\end{theorem}

\begin{proof}
It is enough to show that $\mathcal{C}_{DC}(H)$ coincides with the class of enriched quasi nonexpansive mappings, $\mathcal{C}^{e}_{QNE}(H)$. 

Using the Definition \ref{def5} of enriched nonexpansive mappings, it follows that a mapping $T$ is  {\it enriched quasi nonexpansive} if there exists $b\in[0,\infty)$  such that
\begin{equation} \label{eq3w}
\|b(x-y^*)+Tx-y^*\|\leq (b+1) \|x-y^*\|,\forall x \in X \textnormal{ and } y^*\in Fix\,(T).
\end{equation}
Now let $T\in \mathcal{C}_{DC}(H)$. This means that $T$ satisfies \eqref{eq-DC} with $k\in (0,1)$.

Similarly to the proof of Theorem \ref{p5}, we have
$$
\|x-Tx\|^2=\|x-y^*-(Tx-y^*)\|^2=\langle x-y^*-(Tx-y^*), x-y^*-(Tx-y^*)\rangle>
$$
$$
=\|x-y^*\|^2-2\langle x-y^*, Tx-y^*\rangle+\|Tx-y^*\|^2
$$
and so \eqref{eq-DC}  is equivalent to
$$
\|Tx-y^*\|^2\leq \frac{1+k}{1-k}\cdot \|x-y^*\|^2-\frac{2k}{1-k}\cdot \langle x-y^*,Tx-y^*\rangle
$$
Now, by adding in both sides of the previous inequality the quantity
$$
b^2\cdot \|x-y^*\|^2+2b\langle x-y^*,Tx-y^*\rangle
$$
one obtains the equivalent inequality
$$
\|b(x-y^*)+Tx-y^*\|^2
$$
\begin{equation}\label{destept-1}
\leq \left(b^2+\frac{1+k}{1-k}\right) \|x-y^*\|^2+\left(2b-\frac{2k}{1-k}\right)\cdot \langle x-y^*,Tx-y^*\rangle.
\end{equation}
Now, by taking $b=\dfrac{k}{1-k}$, we have that $b^2+\dfrac{1+k}{1-k}=(b+1)^2$ and therefore the inequality \eqref{destept-1} is equivalent to 
$$
\|b(x-y^*)+Tx-y^*\|\leq (b+1) \|x-y^*\|,\,\forall x,y \in C, 
$$
which shows that $T\in \mathcal{C}^{e}_{QNE}(H)$. 

The converse follows by the fact that all transformations above are equivalent.

\end{proof}

\begin{remark}
Proposition \ref{p9} shows that the class of demicontractive mappings cannot be enlarged by the technique of enriching contractive mappings. 
\end{remark}

But there are many other important classes of contractive mappings in the metric fixed point theory, e.g., the \' Ciri\' c-Reich-Rus contractions and \' Ciri\' c quasi contractions that deserve to be investigated in order to establish whether they are (un)saturated or not.

Let $(X,d)$ be a metric space. In 1971, \' Ciri\' c \cite{Cir71}, Reich \cite{Rei71} and Rus \cite{Rus71} have established independently a fixed point theorem for mappings $T:X\rightarrow X$ satisfying the following condition:
\begin{equation} \label{cond_CRR}
d(Tx,Ty)\leq a d(x,y)+b \left(d(x,Tx)+d(y,Ty)\right),\textnormal{ for all } x,y \in X,
\end{equation}
where $a,b\geq 0$ and $a+2b<1$.

Denote by $\mathcal{C}_{CRR}(X)$ the class of \' Ciri\' c-Reich-Rus contractions on $X$, that is, the class of mappings which satisfy the contractive condition \eqref{cond_CRR} on the metric space $(X,d)$. 

We note that if $b=0$, condition  \eqref{cond_CRR} reduces to Banach's contraction condition \eqref{cond_Banach}, while, for $a=0$ condition  \eqref{cond_CRR} reduces to Kannan's contraction condition \eqref{cond_Kannan}.
\medskip

{\bf Open problem 1.} \label{pr1}
Let $X$  be a Banach space. Is $\mathcal{C}_{CRR}(X)$ a saturated / unsaturated class of contractive mappings ?
\medskip

Denote by $\mathcal{C}_{QC}(X)$ the class of \' Ciri\' c quasi contractions on $X$, that is, of mappings which satisfy the contractive condition \eqref{eq-Ciric} on the metric space $(X,d)$.
\medskip

{\bf Open problem 2.} \label{pr2}
Let $X$  be a Banach space. Is  $\mathcal{C}_{QC}(X)$ a saturated / unsaturated  class of contractive mappings ?

\section{Conclusions and further study}

1. Based on the technique of enriching contractive type mappings $T$ by means of the averaged operator $T_\lambda$,  we introduced the concept of {\it saturated} class of contractive mappings. 

2. We have shown that, from this  perspective, the contractive type mappings in the metric fixed point theory could be separated into two distinct classes, unsaturated contractive mappings and saturated contractive mappings.

3. We identified some important unsaturated classes of mappings, by surveying some significant fixed point results obtained recently for five remarkable classes of contractive mappings, and have shown that the technique of enriching contractive type mappings provided in all cases genuine new fixed point results. 

4. In the second part of the paper, we also identified two important saturated classes of contractive mappings: strictly pseudocontractive mappings and demicontractive mappings. The main feature of these saturated classes of mappings is that they cannot be expanded by means of the technique of enriching contractive mappings. 

5. We also formulated two open problems which are intended to establish whether \' Ciri\' c-Reich-Rus contractions and \' Ciri\' c quasi contractions are saturated or unsaturated classes of contractions. 

6. As further additional study, we are interested  to consider the same problem as above but for other contractive type mappings from metric fixed point theory that are exposed in \cite{Alg}, \cite{Ber07}, \cite{Ber12}, \cite{BerP13}, \cite{Mes}, \cite{Rho}, \cite{Rus79}-\cite{Rus08} and references therein.

\section*{Acknowledgments}
The first author acknowledges the constant support offered by Department of Mathematics, Faculty of Sciences, North University Centre at Baia Mare, Technical University of Cluj-Napoca.

\end{document}